\newif\ifsubsections
\subsectionstrue 

\documentclass[]{amsart}

\usepackage[sc]{mathpazo}          
\usepackage{eulervm}               
\usepackage[mathscr]{eucal}		   
\usepackage[scaled=0.86]{berasans} 
\usepackage[scaled=1]{inconsolata} 
\usepackage{float}                 
\usepackage[font=normalsize,labelfont=bf]{caption} 
\usepackage{enumitem}              
\usepackage{mathtools}             
\usepackage{xspace}					

\usepackage{tikz-cd}
\usepackage{amssymb}
\usepackage[%
	protrusion=true,
	expansion=false,
	auto=false
	]{microtype}

\usepackage{xcolor}
\usepackage{graphicx}
\graphicspath{{./figures/}}

\PassOptionsToPackage{hyphens}{url} 
\usepackage[
    setpagesize=false,
    pagebackref,
	pdfpagelabels,
	plainpages=false,
    pdfstartview={FitH 1000},
    bookmarksnumbered=false,
    linktoc=all,
    colorlinks=true,
    anchorcolor=black,
    menucolor=black,
    runcolor=black,
    filecolor=black,
    linkcolor=blue,
	citecolor=blue,
	urlcolor=red,
]{hyperref}
\usepackage[backrefs,msc-links,initials,nobysame]{amsrefs}

\newtheorem{mylem}{Lemma}[equation]
\newtheorem{mythm}[equation]{Theorem}
\newtheorem{myprop}[equation]{Proposition}

\theoremstyle{definition}
\newtheorem{mydef}[equation]{Definition}
\newtheorem{myexer}[equation]{Exercise}
\newtheorem{myex}[equation]{Example}
\newtheorem{myexs}[equation]{Examples}

\theoremstyle{remark}
\newtheorem{myrmk}[mylem]{Remark}

\DeclareMathOperator{\sep}{sep}

\newcommand{\defi}[1]{\emph{#1}} 
\newcommand{\OO}{{\mathcal O}}
\newcommand{\mm}{{\mathfrak m}}
\newcommand{\A}{{\mathbb A}}
\newcommand{\F}{{\mathbb F}}
\newcommand{\Z}{{\mathbb Z}}
\newcommand{\Q}{{\mathbb Q}}
\newcommand{\R}{{\mathbb R}}
\newcommand{\C}{{\mathbb C}}
\newcommand{\PP}{{\mathbb P}}
\newcommand{\Xbar}{{\overline X}}
\newcommand{\Xsep}{{X^{\sep}}}
\newcommand{\kbar}{{\overline k}}
\newcommand{\Fbar}{{\overline F}}
\newcommand{\Fsep}{{F^{\sep}}}
\newcommand{\ksep}{{k^{\sep}}}

\newcommand{\calA}{{\mathcal A}}
\newcommand{\calB}{{\mathcal B}}
\newcommand{\calC}{{\mathcal C}}

\newcommand{\calV}{{\mathcal V}}

\newcommand{\kk}{{\mathbf k}}

\DeclareMathOperator{\Br}{Br}
\DeclareMathOperator{\Pic}{Pic}
\DeclareMathOperator{\Mat}{M}
\DeclareMathOperator{\inv}{inv}
\DeclareMathOperator{\HH}{H}
\DeclareMathOperator{\et}{et}
\DeclareMathOperator{\Spec}{Spec}
\DeclareMathOperator{\codim}{codim}
\DeclareMathOperator{\im}{im}
\DeclareMathOperator{\Gal}{Gal}

\begin{document}

%
%
%
%
%
%

\title{Rational points on varieties and the Brauer-Manin obstruction}

%
%
\author{Bianca Viray}
\address{{University of Washington, Department of Mathematics, Box 354350, Seattle, WA 98195, USA}}
\email{bviray@uw.edu}
%
%





%
%
\maketitle

%
%


\section*{Prologue}

	A central object of interest in arithmetic geometry is the set of \(k\)-rational points on a smooth projective geometrically integral \(k\)-variety \(X\).  In these notes, we focus on the fundamental problem of determining whether this set, denoted \(X(k)\), is nonempty.

	If you are handed a variety, e.g., \(X := V(x^3 + 2y^3 + 10z^3)\subset \PP^2\), and asked to determine if it has rational points, a natural first step is to try to find some solutions where the coordinates are small (in other words, solutions of \defi{small height}).  If you are lucky, you can find a point of small height (try to do so with the example above!), and you will thereby have proved that the set \(X(\Q)\neq\emptyset\).

	In fact, if \(X(k)\neq \emptyset\), then an approach like this will give a proof of nonemptyness! Any point \(P\in X(k)\) has a (finite) \defi{height} \(H\), and searching in a box of bounded height \(\{x\in \PP^n(k) : H(x)\leq B\}\) is a finite procedure.  Therefore, increasing \(B\) by constant  amount and searching in these larger boxes will (eventually!) result in you finding the point \(P\).

	However, if \(X(k) = \emptyset\), then this approach will {never} terminate.  Indeed, if your search keeps failing, you cannot tell whether you need to search in an even larger box, or whether actually there is no rational point.  Searching \emph{cannot} certify an absence \(k\)-points; for that, we need another method.  These notes focus on one such method that is known as the \defi{Brauer-Manin obstruction}.

	\subsection*{The goal of these notes}
		There are many excellent references on the Brauer-Manin obstruction (e.g.,~\cites{Skorobogatov-Torsors, Poonen-Qpts, CTS-BrauerBook} just to name a few).  We do not endeavor to improve on (or replicate!) those references here. Rather, these notes should be thought of as a  guidebook to the field, giving an overview of the current landscape, with pointers on where to go if you, the reader, wish to explore more.  Just as travel guidebooks have particular biases (affordability, nature, good for kids, etc.), so do these notes.  They give \emph{my} perspective on how the feedback loop between computation and theory currently manifests in the study of rational points and the Brauer-Manin obstruction.

	\subsection*{Acknowledgements}
		These notes were prepared to accompany a lecture series at the 2022 Park City Mathematics Institute 2022 Graduate Summer School. My perspective on this subject, particularly that presented in Sections~\ref{sec:Capturing} and~\ref{sec:Persistence}, has been greatly influenced by my collaborations with Brendan Creutz, which I have been honored to take part in.

        I thank the organizers of the PCMI 2022 Graduate Summer School, Jennifer Balakrishnan, Bjorn Poonen, and Akshay Venkatesh, for the opportunity to speak, and I thank the PCMI Director, Rafe Mazzeo, and the PCMI Program Manager, Dena Virgil, for their assistance and for running a wonderful program.  These notes were greatly improved by comments and questions from the graduate student participants at the PCMI summer school; I am grateful for their interest, engagement, and patience with an early draft of these notes.  Thomas Carr and Carlos Rivera skillfully served as TA's for my lecture series; I thank them for finding many typos and improving the exposition.  I also thank Alex Galarraga, Ting Gong, Leo Mayer, Haoming Ning, Caelan Ritter, Alex Wang, and Olivier Wittenberg for their careful readings and comments, which have improved the final version.

	\subsection*{Notation and conventions}

		Throughout, we will use \(F\) to denote an arbitrary field.  We write \(\Fsep\) for a fixed separable closure, and \(\Fbar\) for a fixed algebraic closure containing \(\Fsep\).  We use \(G_F\) to denote the absolute Galois group \(\Gal(\Fsep/F)\).
		
		We reserve \(k\) to denote a number field.  We write \(\Omega_k\) for the set of places of \(k\), and for any \(v\in\Omega_k\) we let \(k_v\) denote the completion.  If \(v\in\Omega_k\) is nonarchimedean, then we write \(\OO_v\subset k_v\) for the valuation ring, \(\mm_v\) for the maximal ideal, \(\F_v := \OO_v/\mm_v\) for the residue field, and \(k_v^{nr}\) for the maximal unramified extension of \(k_v\).  We write \(\A_k := \prod'(k_v,\OO_v)\) for the adele ring of \(k\), i.e., the restricted product \(\{(x_v) \in \prod_v k_v : \#\{v : x_v\notin \OO_v\}<\infty \}\).  
		
		We will use \(X\) to denote a smooth projective geometrically integral variety over \(k\).  For any extension \(F/k\) we write \(X_F\) for the base change of \(X\) to \(\Spec F\). We also will use the conventions \(\Xbar := X_{\kbar}\) and \(\Xsep:= X_{\ksep}\).

\section{Obstructions to the existence of rational points}

	Given a smooth projective geometrically integral variety \(X\) over a number field \(k\), there is no known general method to determine if the set \(X(k)\) is nonempty.\footnote{In fact, it may be an undecidable problem! This question of decidability is known as Hilbert's tenth problem; see~\cite{Poonen-H10Notices} for more details.}  Instead, we find a set \(S\) that 1) contains \(X(k)\) and 2) seems to be computable (at least in some cases).  Then, if we can prove that \(S=\emptyset\), this will imply that \(X(k) = \emptyset\).  The set \(S\) is typically called an \defi{obstruction set} since the emptyness of \(S\) \emph{obstructs} the existence of \(k\)-points.

	\subsection{Local obstructions}
		The simplest obstruction set comes from the embeddings of \(k\) into one of its completions \(k_v\).  Precisely, if \(v\) is a place of \(k\), then the embedding \(k\hookrightarrow k_v\) induces an inclusion \(X(k)\subset X(k_v)\).  So if there are no \(k_v\)-points on \(X\), then there are also no rational points. In addition, if \(v\) is nonarchimedean, then Hensel's lemma can be used to show that there is some \(n\) such that \(X(k_v)\neq \emptyset\) if and only if \(X(\OO_v/\mm_v^n)\neq\emptyset\)~\cite{Nerode}.  In particular, it is a finite computation to determine if \(X(k_v) \neq\emptyset\) for nonarchimedean places \(v\) (and the same is true for archimedean places, see~\cite{Poonen-Qpts}*{Remark 2.6.4}).  (Note, analogous effectivity results are not known for local fields of positive characteristic)

		\begin{myexer}Show that
			\[
				X := V(x^2 + y^2 + 7z^2)\subset \PP^2
			\]
			has no \(\Q_7\)-points.  Conclude that \(X(\Q) = \emptyset\).
	\end{myexer}
	
        We may package these obstructions together using the ad\`eles \(\A_k\) of \(k\).  Just as above, the embedding \(k\hookrightarrow\A_k\) allows us to view \(X(k)\) as a subset of \(X(\A_k)\), and one can prove (see~\cite{Poonen-Qpts}*{Section 2.6.3}) that
        \[
            X(\A_k) = \prod_{v\in\Omega_k} \!\!'  (X(k_v), X(\OO_v)). 
        \]
        Since, by assumption, \(X\) is projective, the valuative criterion for properness implies that the inclusion \(X(\OO_v)\subset X(k_v)\) is an equality.  Hence, \(X(\A_k) = \prod X(k_v)\).  
		
		It turns out that even though \(X(\A_k)\) encapsulates solubility conditions at the infinitely many completions at once, it is still a finite computation to determine whether \(X(\A_k)\neq\emptyset\).  
		\begin{myprop}[c.f.~\cite{Poonen-Qpts}*{Thm. 7.7.2}]\label{prop:LocalSolubilityFiniteComputation}
			Let \(X\) be a smooth projective geometrically integral variety over a number field \(k\).  Then the set of places where \(X\) fails to have \(k_v\)-points is finite and is contained in an effectively computable set.
		\end{myprop}
        This proposition implies that given \(X\) one can effectively compute a finite set \(S\subset \Omega_k\) such that \(X(k_v)\neq\emptyset\) for all \(v\notin S\).  Then for each \(v\in S\), it is a finite computation to determine whether \(X(k_v) = \emptyset\) (see the first paragraph of this section).  Since \(S\) is finite, this proves our assertion that determining whether \(X(\A_k)\neq\emptyset\) is a finite computation.
		\begin{proof}[Proof of Proposition~\ref{prop:LocalSolubilityFiniteComputation}]
			Since \(X\) is smooth and geometrically integral, then by slicing with sufficiently general hyperplanes, we may obtain a smooth, geometrically integral curve \(C\subset X\).  By generic smoothness, \(C\) necessarily has good reduction away from a finite set of places (see~\cite{Poonen-Qpts}*{Section 3.2} for more details).  Furthermore, the Hasse-Weil bounds~\cite{Poonen-Qpts}*{Cor. 7.2.1} imply that, for any place of good reduction with \(\#\F_v\) sufficiently large, the reduction \(C\bmod v\) is guaranteed to have a smooth \(\F_v\)-point.  Hensel's lemma~\cite{Poonen-Qpts}*{3.5.63} then implies that \(C\), and hence \(X\), has \(\Q_v\)-points for such \(v\).  Thus we have proved that \(X(k_v)\neq\emptyset\) for all \(v\) outside the following subset
			\[
			\left\{v|\infty\right\}\cup \left\{C \bmod v \textup{ singular}\right\}	\cup\left\{\#\F_v < N\right\},
			\]
			where \(N\) is some positive integer given by the Hasse-Weil boards.  Since this set is finite and effectively computable, we have proved the desired result.
		\end{proof}
		\begin{myexer}
			\textit{Note:} For this problem, the following specific consequence of Hensel's Lemma will be useful:  If \(p\) be a prime and \(u\in \Z_p^{\times}\), then
			\[
				u \in \Z_p^{\times2}\Leftrightarrow \begin{cases}
					u\bmod p \in \F_p^{\times2} & \textup{if }p\neq 2,\\
					u \equiv 1 \pmod 8 & \textup{if }p= 2.
				\end{cases}
			\]\label{exer:ConicLocal}
			
			\begin{enumerate}
				\item Let \(p\) be an odd prime and let \(a,b,c\in \Z - p\Z\).  Show that \(\{ax^2 : x\in \F_p\}\) and \(\{c - by^2 : y \in \F_p\}\) both have cardinality \(\frac{p+1}{2}\) and that, therefore, the sets contain a common value.  Use this result to show that
				\[
					X := V(ax^2 + by^2 + cz^2)\subset \PP^2
				\]
				has a \(\Q_p\)-point.
				\item Determine whether \(X := V(5x^2 + 7y^2 - 3z^2)\subset \PP^2\) has \(\A_{\Q}\)-points.
				\item Let \(a,b,c\in \Z\) be squarefree, pairwise relatively prime integers.  Prove that \(X := V(ax^2 + by^2 + cz^2)\) has \(\A_{\Q}\)-points if and only if \(a,b,c\) are not all the same sign and \(ax^2 + by^2 + cz^2\) has solutions in \(\Z/8abc\Z\) with the property that for every \(p|8abc\), at least two of the coordinates are nonzero modulo \(p\).

			\end{enumerate}
		\end{myexer}

		\subsubsection{The local-to-global principle}
			\begin{mydef}
				A class of varieties \(\calC\) satisfies the \defi{local-to-global principle} if, for all \(X\in \calC\), 
				\[
					X(\A_k) \neq\emptyset \Leftrightarrow X(k)\neq \emptyset.	
				\]
			\end{mydef}
			\begin{myexs}\item\label{ex:LGP}
				\begin{enumerate}
					\item Quadrics, in any number of variables, satisfy the local-to-global principle by the Hasse-Minkowski theorem.
					\item \defi{Severi-Brauer varieties}, i.e., varieties that are geometrically \emph{isomorphic} to projective space, satisfy the local-to-global principle by the Albert-Brauer-Hasse-Noether theorem.\label{part:SBvar}
				\end{enumerate}
			\end{myexs}

			Note that every variety in one of the two classes in the above example is geometrically rational, i.e., over \(\kbar\) there is a birational map \(\Xbar\dasharrow \PP^n_{\kbar}\).  Classes of varieties with more complicated geometry often have members that fail the local-to-global principle.  The first example of this failure dates back to the 1940's and is due to Lind and Reichardt, independently.

			\begin{mythm}[\cites{Lind, Reichardt}]
				Let \(C\subset \PP^3\) be the smooth projective genus \(1\) curve defined by the two quadrics
				\(
				 	2y^2 = 	w^2 - 17z^2\) and \(wz = x^2.\)
				Then \[C(\A_{\Q}) \neq \emptyset\quad\textup{and}\quad C(\Q) = \emptyset.\]
			\end{mythm}

			To prove that \(C\) has no rational points despite the presence of the adelic points, we need a \emph{refined} obstruction set, that is, an intermediate set \(S\) that sits between the set of rational points and the set of adelic points.

	\subsection{An introduction to the Brauer-Manin obstruction}
		
		The goal of this section is to define the Brauer-Manin set \(X(\A_k)^{\Br}\), and show that it is a refined obstruction set, i.e., that we have the following containments:
		\[
			X(k)\subset X(\A_k)^{\Br}\subset X(\A_k).	
		\]
		To do so, we must first introduce the Brauer group.

		\subsubsection{The Brauer group of a field} 

			Let \(F\) be a field.
			\begin{mydef}\hfill
				\begin{enumerate}
				\item A \defi{central simple algebra} over \(F\) is a finite dimensional \(F\)-algebra \(\calA\) whose center is exactly \(F\) and that has no nontrivial two-sided ideals.
				\item Two central simple algebras \(\calA, \calB\) over \(F\) are said to be \defi{Brauer equivalent} if there exist positive integers \(n,m\) such that \(\calA\otimes_F \Mat_m(F)\simeq\calB\otimes_F\Mat_n(F)\).
				\end{enumerate}
			\end{mydef}
			\begin{mythm}[\cite{GS-CSA}*{Prop. 2.4.8 and Thm. 4.4.7}]\label{thm:BrF}
				Let \(F\) be a field.  The set of Brauer equivalence classes of central simple algebras over \(F\) forms a torsion abelian group under tensor product, where the identity element is the class of \(\Mat_n(F)\). Moreover, this group is isomorphic to the Galois cohomology group \(\HH^2(G_F, (F^{\textup{sep}})^{\times})\).
			\end{mythm}
			\begin{mydef}
				The group from Theorem~\ref{thm:BrF} is known as the \defi{Brauer group of \(F\)} and is denoted \(\Br  F\).
			\end{mydef}

			\begin{myex}\label{ex:Quaternion}
				Assume that \(F\) has characteristic different from \(2\) and let \(a,b\in F^{\times}\).  Then the \defi{quaternion algebra}
				\[
					\calA_{a,b} := 	F\otimes F\cdot i \otimes F\cdot j \otimes F \cdot ij, \quad i^2 = a, j^2 = b, ji = -ij
				\]
				is a central simple algebra over \(F\).  It has order dividing \(2\) in \(\Br  F\)~\cite{GS-CSA}*{Cor. 1.5.3}.
			\end{myex}
			\begin{myexer}\label{exer:TrivialBrauerClass}
				Let \(F\) be a field of characteristic different from \(2\) and let \(a,b\in F^{\times}\).  
        
        		Prove that \(\calA_{a,b}\simeq \Mat_2(F)\) if and only if there is some \(x,y,z\in F\), not all zero, such that \(ax^2 + by^2 = z^2\).
    
				Conclude that \( (a,b) := [\calA_{a,b}] \in \Br  F\) is trivial if and only if the conic \(C_{a,b}:ax^2 + by^2 = z^2\) has an \(F\)-rational point, and that  \(C_{a,b}(F) \neq\emptyset\) if and only if \(a\in \textup{N}(k(\sqrt{b})^{\times})\).  (By symmetry this is equivalent to \(b\in \textup{N}(k(\sqrt{a})^{\times})\)).
			\end{myexer}
    
			\begin{myrmk}\label{rmk:ProductsOfQuaternions}
				One can also show that \(\calA_{a,b}\otimes_k \calA_{a,c} \simeq \Mat_2(\calA_{a,bc}),\) (see~\cite{GS-CSA}*{Lemma 1.5.2}) which implies that in \(\Br  F,\) we have \((a,b)(a,c) = (a,bc)\).
			\end{myrmk}

			The above example of quaternion algebras is a special case of a general correspondence.  To any central simple algebra over \(F\), one can associate a Severi-Brauer variety (see Example~\ref{ex:LGP}\eqref{part:SBvar}), and vice versa~\cite{GS-CSA}*{Section 5.2}.

			Our interest in the Brauer group stems from the fact that, over a number field \(k\), the Brauer group encodes the abelian reciprocity laws of \(k\).  This is encapsulated by the fundamental exact sequence of global class field theory.
			\begin{mythm}[\cite{Poonen-Qpts}*{Thms. 1.5.34 and 1.5.36}]\label{thm:CFT}
				Let \(k\) be a number field.  For each place \(v\), there is an injective homomorphism
				\[
					\inv_v\colon \Br  k_v \to \Q/\Z,	
				\]
				that is an isomorphism for nonarchimedan \(v\).  For archimedean \(v\), the image is \(\frac12\Z/\Z\) if \(k_v = \R\) and \(0\) if \(k_v = \C\).  Furthermore, these isomorphisms fit together in the following short exact sequence
				\begin{equation}\label{eq:BrExact}
					0 \to \Br  k \to \oplus_v \Br  k_v \xrightarrow{\sum_v \inv_v}\Q/\Z\to 0.	
				\end{equation}
			\end{mythm}

			\begin{myrmk}
				When working with explicit computations, one has to take care that the maps \(\inv_v\) are defined in a globally compatible way.  An arbitrary collection of isomorphisms \(\phi_v\colon \Br  k_v\to \Q/\Z\) for nonarchimedean \(v\) will not necessarily give an exact sequence as in~\eqref{eq:BrExact}.  See~\cite{CTS-BrauerBook}*{Definition 13.1.7} for a definition of \(\inv_v\).  The explicit examples in these notes will restrict to \(2\)-torsion elements to avoid this subtlety.
			\end{myrmk}

			\begin{myex}\label{ex:QR-Br}
				Let \(a,b\in k^{\times}\) and let  \(v\) be a place of \(k\).  Using Example~\ref{ex:Quaternion}, one can check that \(\calA_{a,b}\in \Br  k_v\) is nontrivial if and only if the Hilbert symbol \((a,b)_v = -1\).  (Indeed, the definition of the Hilbert symbol is that \( (a,b)_v = -1\) if \(C_{a,b}(k_v) = \emptyset\) and \((a,b) _v= 1\) otherwise.)  Thus, the fact \(\sum_v \inv_v(\calA_{a,b}) = 0\) is exactly the product formula 
				\[
					\prod_v(a,b)_v = 1,	
				\]
				which, over \(\Q\), is also equivalent to quadratic reciprocity.
			\end{myex}
            \begin{myexer}\label{exer:QRandProductFormula}\hfill
                \begin{enumerate}
                    \item Let \(p\) be an odd prime and let \(a\) be an integer that is relatively prime to \(p\).  Prove that the Hilbert symbol \((a,p)_p\) equals the Legendre symbol \(\left(\frac{a}p\right)\).
                    \item Let \(p,q\) be distinct odd primes.  Prove \((p,q)_2 = (-1)^{\frac{(p-1)(q-1)}{4}}\).
                    \item Let \(p,q\) be distinct odd primes. Using Exercise~\ref{exer:ConicLocal}, prove that
                    \[
                        \prod_{v}(p,q)_v = (p,q)_p(p,q)_q(p,q)_2 = \left(\frac{q}p\right)\left(\frac{p}q\right)(-1)^{\frac{(p-1)(q-1)}{4}}.
                    \]
                    \item Let \(p\) be an odd prime.  Following the same strategy as above, prove that 
                    \[
                        \prod_{v}(2,p)_v  = \left(\frac{2}p\right)(-1)^{\frac{(p^2-1)}{8}}.
                    \]
                    \item  Let \(p\) be an odd prime.  Following the same strategy as above, prove that 
                    \[
                        \prod_{v}(-1,p)_v  =\left(\frac{-1}p\right)(-1)^{\frac{(p-1)}{2}}
                    \]
                    \item Prove that for all \(a,b\in \Q^{\times}\), \(\prod_v (a,b)_v = 1\) if and only if for all pairs of distinct odd primes \(p,q\), we have
                    \[
                        \prod_v (p,q)_v = 1, \quad
                        \prod_v (2,p)_v = 1, \quad
                        \prod_v (-1,p)_v = 1.
                    \]
                    \item Combine the previous parts to observe that the product formula over \(\Q\) is equivalent to quadratic reciprocity.
                \end{enumerate}
            \end{myexer}

			For more extensive introductions on the Brauer group of a field, see \citelist{\cite{GS-CSA}\cite{CTS-BrauerBook}*{Chap. 1} \cite{Milne-EC}*{Chap. 4}}.
		\subsubsection{The Brauer group of a variety}
			The notion of the Brauer group of a field can be generalized to the Brauer group of a scheme using \'etale cohomology.  
			\begin{mydef}
				The Brauer group of a scheme \(X\) is \(\Br  X := \HH^2_{\et}(X, \mathbb{G}_m)\).\footnote{There are several ways to extend the definition of the Brauer group from a field to a scheme.  The definition we use here is sometimes called the \defi{cohomological Brauer group}, since, in complete generality, it is not necessarily equal to other possible definitions of the Brauer group of a scheme.  However, under the standing assumptions of these notes (i.e., smooth projective varieties over a field of characteristic \(0\)), all of the generalization coincide, so we will simply refer to this as the Brauer group.}
			\end{mydef}

			For our purposes, the following properties will be particularly useful.
			\begin{mythm}\label{thm:BrProps}\hfill
				\begin{enumerate}
					\item For any field \(F\) we have \(\Br  F = \Br\, \Spec F\).
					\item The Brauer group is functorial, i.e., if there is a morphism of varieties \(f\colon X\to Y\), then we have a homomorphism \(f^*\colon \Br  Y \to \Br  X\).\label{part:functoriality}
					\item  If \(X\) is a smooth geometrically integral variety over a field \(F\) of characteristic \(0\), then the inclusion of the generic point \(\iota\colon \Spec \kk(X) \to X\) gives an injective homomorphism \(\iota^*\colon \Br  X\hookrightarrow\Br  \kk(X)\).
					\item The Brauer group is a birational invariant of smooth projective varieties, i.e., if \(f\colon X\dashrightarrow Y\) is a birational map between two smooth projective varieties, then \(f^*\colon \Br  \kk(Y)\to \Br  \kk(X)\) sends \(\Br  Y\) isomorphically to \(\Br  X\).
					\item If \(X\) is a smooth projective variety and let \(U\subset X\) be an open subset such that \(\codim(X\setminus U)\geq 2\), then \(\Br  X = \Br  U\).\label{part:codim2}
				\end{enumerate}
			\end{mythm}
			\begin{proof}
				The first two properties follow from the definition together with properties of \'etale cohomology.  For the last three properties, see~\cite{CTS-BrauerBook}*{Thm. 3.5.5},~\cite{CTS-BrauerBook}*{Cor. 6.2.11}, and~\cite{CTS-BrauerBook}*{Thm. 3.7.6}, respectively.  
			\end{proof}

			See~\citelist{\cite{CTS-BrauerBook}\cite{Poonen-Qpts}*{Section 6.6}} for more details on the Brauer group of a scheme.

		\subsubsection{An example of a Brauer-Manin obstruction using quadratic reciprocity}

			Before introducting the full Brauer-Manin obstruction in the next section, we first give an extended treatment of a particular type of Brauer-Manin obstruction using conics.

			Let \(X\) be a smooth projective variety over \(\Q\) and assume that we have a conic bundle \(\pi\colon \calC\to X\) over \(X\) where the morphism is smooth.  Then for every extension \(F/\Q\), we have a set map
			\[
				X(F)\to \{\textup{conics over }F\}, \quad x\mapsto \calC_x := \pi^{-1}(x).	
			\]
			In particular, by applying this to \(F = \Q\) and \(F = \Q_v\) for all places \(v\), we obtain the following commutative diagram.
			\begin{equation}\label{diag:ConicsTake0}
				\begin{tikzcd}
					X(\Q) \arrow[r, hook] \arrow[d] & X(\A_{\Q}) = \prod_v X(\Q_v) \arrow[d] \\
					\{\textup{conics over }\Q\} \arrow[r, hook]  & \prod_v\{\textup{conics over }\Q_v\},
				\end{tikzcd}
			\end{equation}
			where the bottom inclusion is given by diagonal basechange. 

			To see how we can use the above diagram to define an obstruction set we must first explore the following rephrasing of quadratic reciprocity that we alluded to in Example~\ref{ex:QR-Br}.
			\begin{mythm}[Rephrasing of quadratic reciprocity I]\label{thm:QRI}\hfill

			\noindent Let \(C: V(ax^2 + by^2 = z^2)\subset \PP^2_{\Q}\) be a smooth conic over \(\Q\).  Then
				\[
					\{v\in \Omega_{\Q}: C(\Q_v) = \emptyset\}	
				\]
				is a finite set with \emph{even} cardinality.
			\end{mythm}
			The finiteness is a special case of Proposition~\ref{prop:LocalSolubilityFiniteComputation}; the surprising fact is that the cardinality is always even!  Why is this surprising?  Let's rephrase the theorem again.
			\begin{mythm}[Rephrasing of quadratic reciprocity II]\label{thm:QRII}
				\hfill

			\noindent
				Let \(C: V(ax^2 + by^2 = z^2)\subset \PP^2_{\Q}\) be a smooth conic over \(\Q\).  Then
				\[
					\#\{p \textup{ prime}: C(\Q_p) = \emptyset\} \equiv 1 \bmod 2\; \Longleftrightarrow \;a<0 \textup{ and } b<0.
				\]
			\end{mythm}
			\begin{myex} 
				Verify that Theorems~\ref{thm:QRI} and~\ref{thm:QRII} are indeed equivalent.
			\end{myex}
			Let's unravel the statement of Theorem~\ref{thm:QRII}.  It says that if you know whether \(C\) is soluble over \(\Q_p\) for \emph{all} primes \(p\), that tells you something about the positivity of \(a\) and \(b\).  This is a strong and surprising statement! There is \emph{no} notion of positivity in the \(p\)-adics.  Positivity is a concept from the real numbers.  But quadratic reciprocity implies that \(p\)-adic information for \emph{all} the primes \(p\) {can} detect some positivity information. 
            \texttt{:exploding\_head:}

            \medskip

			Now we return to the commutative diagram~\eqref{diag:ConicsTake0} and apply Theorem~\ref{thm:QRI}.  Then we obtain the following.
			\[
				\begin{tikzcd}
					X(\Q) \arrow[rr, hook] \arrow[d] && \prod_v X(\Q_v) \arrow[d] \\
					\{\textup{conics}/\Q\} \arrow[r, hook] 
					& 
					\left\{(C_v) : 
					\#\{v : C_v(\Q_v)=\emptyset\}\equiv 0 \bmod 2
					\right\} \arrow[r, hook]
					 & \prod_v\{\textup{conics}/\Q_v\}
				\end{tikzcd}
			\]
			The intermediate set on the bottom row provides a construction of an intermediate set on the top row.
			\[
				\begin{tikzcd}
					X(\Q) \arrow[r, hook] \arrow[d] &
					\left\{(x_v) : \#\{v : \calC_{x_v}(\Q_v)=\emptyset\}\equiv 0 \bmod 2\right\}\arrow[r, hook]\ar[d]
					&  \prod_v X(\Q_v) \arrow[d] \\
					\{\textup{conics}/\Q\} \arrow[r, hook] 
					& 
					\left\{(C_v) : 
					\#\{v : C_v(\Q_v)=\emptyset\}\equiv 0 \bmod 2
					\right\} \arrow[r, hook]
					 & \prod_v\{\textup{conics}/\Q_v\}
				\end{tikzcd}
			\]
			And thus we obtain a refined obstruction set!

			To conclude this section, let us see how this can be phrased in terms of the Brauer group.  By Example~\ref{ex:Quaternion} and Exercise~\ref{exer:TrivialBrauerClass}, we have seen that we can associate a Brauer class to a conic.  Further, we saw in Example~\ref{ex:QR-Br} that quadratic reciprocity is encoded by the \(2\)-torsion in the fundamental exact sequence of global class field theory.  Thus, we may express the top row of the above diagram in the following, equivalent, way:
			\[
				\begin{tikzcd}
					X(k) \arrow[r, hook]  &
					\left\{(x_v) :\sum_v \inv_v[\calC_{x_v}]=0\in \Q/\Z\right\}\arrow[r, hook]
					& X(\A_k) = \prod_v X(k_v).
				\end{tikzcd}
			\]
			In fact, we can work with the Brauer group from the beginning.  Given a conic bundle \(\pi\colon \calC \to X\) where the morphism is smooth, one can associate a Brauer class \(\calA := [\calC]\in \Br  X\).  Further, given an extension \(F/k\) and a map \(x\colon \Spec F \to X\) (equivalently a point \(x\in X(F)\)), the Brauer class \(x^*\calA\) coming from Theorem~\ref{thm:BrProps}\eqref{part:functoriality} is exactly the class \([\calC_x]\) given by the fiber of \(\pi\) above \(x\).

		\subsubsection{The Brauer-Manin set}
			In the previous section, we saw how a conic bundle \(\pi\colon \calC\to X\) carves out a refined obstruction set and how this can be phrased in terms of the Brauer group.  In this section, we will show how an arbitrary element of the Brauer group of a variety \(X\) (i.e., not just those Brauer classes corresponding to conic bundles) carves out a subset of \(X(\mathbb{A}_k)\) that contains the \(k\)-rational points.  This subset was introduced by Manin~\cite{Manin-ICM} and is known as the \defi{Brauer-Manin set}.

			For any extension \(F/k\), the functoriality of the Brauer group (Theorem~\ref{thm:BrProps}\eqref{part:functoriality})
			gives a pairing
			\[
				\Br  X\times X(F) \to \Br  F, (\alpha, x) \mapsto \langle\alpha, x\rangle := x^*\alpha,	
			\]
			where we view the point \(x\in X(F)\) as a map \(x\colon \Spec F \to X\).

			If \(X\) is projective, then \(X(\mathbb{A}_k) = \prod_v X(k_v)\) and so we may apply these pairings componentwise to obtain a pairing
			\[
			\Br  X\times X(\mathbb{A}_k) \to \prod_v \Br  k_v.	
			\]
			Using integral models and properties of Brauer groups of local fields, one can show that the image of this pairing actually lands in \(\bigoplus_v\Br  k_v\), i.e., that for every \(\alpha\in \Br  X\) there exists a finite set \(S_{\alpha}\subset \Omega_k\) such that \(\langle\alpha, -\rangle\colon X(k_v)\to\Br  k_v\) is identically \(0\) for all \(v\notin S_{\alpha}\)~\cite{Poonen-Qpts}*{Prop. 8.2.1}.  Therefore, we have the following commutative diagram.
			\[
			    \begin{tikzcd}
					\Br  X\times X(k) \arrow[r, hook] \arrow[d] & \Br  X\times X(\A_k) \arrow[d] \\
					\Br  k \arrow[r]  & \oplus_v\Br  k_v
				\end{tikzcd}
			\]
			Recall from Theorem~\ref{thm:CFT}, the bottom vertical arrow fits into an exact sequence, so we may extend our diagram as follows.
			\[
			    \begin{tikzcd}
					& \Br  X\times X(k) \arrow[r, hook] \arrow[d] & \Br  X\times X(\A_k) \arrow[d] \arrow[rrd]\\
					0 \arrow[r]&\Br  k \arrow[r]  & \oplus_v\Br  k_v\arrow[rr, "\sum\inv_v"]&& \Q/\Z\arrow[r] & 0
				\end{tikzcd}
			\]
			Observe that since the bottom row is a complex and the diagram commutes, then the induced map \(\Br  X\times X(k)\to \Q/\Z\) is identically \(0\).  In particular, for all \(\alpha\in \Br  X\), the set \(X(k)\) is contained in the set of \defi{adelic points orthogonal to \(\alpha\)}, i.e.,
			\[
				X(k)\subset X(\A_k)^{\alpha} := \{(x_v)\in X(\A_k) : \langle\alpha, (x_v)\rangle = 0\}.
			\]
			Taking the intersection \(X(\A_k)^{\alpha}\) for all \(\alpha\), we obtain the \defi{Brauer-Manin set}
			\begin{equation}\label{eq:DefOfBM}
				X(\A_k)^{\Br} := \bigcap_{\alpha\in \Br  X}X(\A_k)^{\alpha} = \left\{(x_v)\in X(\A_k) : \langle\alpha, (x_v)\rangle = 0\,\forall\alpha\in\Br  X\right\}.
			\end{equation}

			\begin{myexer}\label{exer:BrmodBr0}
				 Let \(X\) be a smooth projective geometrically integral variety over a number field \(k\) and let \(\pi\) denote the structure morphism \(\pi\colon X\to \Spec k\).  
				\begin{enumerate}[itemsep=12pt]
					\item Let \(\alpha_0\in \Br  k\).  Show that \(X(\A_k)^{\pi^*\alpha_0} = X(\A_k)\).
					\item Let \(\alpha,\beta\in \Br  X\).  Show that
					\[
						X(\A_k)^{\alpha}\cap X(\A_k)^{\beta} = \bigcap_{\gamma \in \{\alpha^i\beta^j : i,j\in \Z\}} X(\A_k)^{\gamma}
					\]
				\end{enumerate}			
			\end{myexer}
			These exercises show that the Brauer-Manin set depends only on the quotient \(\Br  X/\im(\pi^*\colon \Br  k \to \Br  X)\).  We write \(\Br_0 X := \im(\pi^*\colon \Br  k \to \Br  X)\) and refer to elements in \(\Br_0 X\) as \defi{constant Brauer classes}. 

		\section{Computing the Brauer-Manin obstruction}

			We have successfully defined a refined obstruction set, the Brauer-Manin set
			\[
				X(k)\subset X(\A_k)^{\Br}\subset X(\A_k),	
			\]
			but we have yet to see whether this is \emph{useful}.  Utility has a theoretical component and a practical or computational component.
			\begin{enumerate}
				\item {[Theory]} Do there exist varieties where \(X(\A_k)^{\Br} = \emptyset\) yet \(X(\A_k)\neq \emptyset\)?
				\item {[Practice/Computation]} Can we compute the Brauer-Manin set?
			\end{enumerate}

			The answer to the first question is a resounding \emph{YES!} In fact, there are many such examples.  When Manin introduced this obstruction, he showed that several of the known failures of the local-to-global principle could be explained by the Brauer-Manin obstruction.  Since then, several more examples have been constructed, which together lead to the expectation that if a class of varieties can have nontrivial Brauer group, it is likely that there exists a variety in that class with a Brauer-Manin obstruction to the existence of rational points.\footnote{Unless there is an ``obvious'' reason why not, e.g., some classes of varieties always have a rational point, like in the case of del Pezzo surfaces of degree \(1\).}

			The answer to the second question is more mixed.  On the one hand, the Brauer-Manin set \emph{has} been computed in several examples, and for some classes of varieties, it is more or less standard to do so (e.g., conic bundles over \(\mathbb{P}^1\)~\cite{CTS-BrauerBook}*{Prop. 11.3.4}). On the other hand, there is no general effectivity result for the Brauer-Manin set, nor is there an approach that is known to work in full generality.  The best effectivity results to date are due to Kresch and Tschinkel~\cites{KT-effectivityAlg, KT-effectivitySurfaces}.  They prove that if \(\Pic \Xbar\) is torsion free, then the \defi{algebraic Brauer-Manin set} \(X(\A_k)^{\Br_1}\) is effectively computable.  The algebraic Brauer-Manin set is, by definition, the set of adelic points orthogonal to those Brauer classes that become trivial after passage to the algebraic closure; this can be (and often is) larger than the Brauer-Manin set. Kresch and Tschinkel can improve their results for surfaces. In that case (and still under the assumption that \(\Pic\Xbar\) is torsion-free), they prove that for any positive integer \(n\), there is an effectively computable set \(X_n\) such that
			\[
				X(\A_k)^{\Br}\subset X_n\subset X(\A_k)^{\Br  X[n]}.
			\]
			In particular, if there exists an effective bound on the exponent of \(\Br  X/\Br_0 X\), then this would imply that the Brauer-Manin set is effectively computable for surfaces with torsion-free geometric Picard group.

			Despite the lack of general algorithms for computing the Brauer-Manin obstruction, there is a general framework that is often helpful for at least computing the group structure of \(\Br  X/\Br_0 X\) as an abstract torsion abelian group.

		\subsection{The Hochschild-Serre spectral sequence and a filtration of the Brauer group}
		We leverage the Hochschild-Serre spectral sequence in \'etale cohomology (applied to the Galois cover \(\Xbar\to X\) and the sheaf \(\mathbb{G}_m\)):
			\begin{equation}\label{eq:HS}
				\HH^p\left(G_k, H^q_{\et}(\Xbar, \mathbb{G}_m)\right)\Longrightarrow
				H^{p+q}_{\et}(X, \mathbb{G}_m).
			\end{equation}
			to give a filtration of the Brauer group.  The exact sequence of low degree terms (see~\cite{Poonen-Qpts}*{Prop. 6.7.1}) is
			\[
				0 \to \Pic X \to (\Pic \Xbar)^{G_k}\to \Br  k\to\ker\left(\Br   X \to \Br  \Xbar\right) \to \HH^1(G_k, \Pic \Xbar) \to 0;
			\]
			(here we use the assumption that \(k\) is a number field and so \(\HH^3(G_k, \kbar^{\times}) = 0\)~\cite{Poonen-Qpts}*{Remark 6.7.10} to obtain the rightmost \(0\)).
			Using the definition of \(\Br_0X\) (see the paragraph after Exercise~\ref{exer:BrmodBr0}), we obtain a short exact sequence
			\begin{equation}\label{eq:Br1}
				0 \to \Br_0 X\to \Br_1 X := \ker\left(\Br   X \to \Br  \Xbar\right) \to \HH^1(G_k, \Pic \Xbar) \to 0.
			\end{equation}
			The subgroup \(\Br_1 X\) is called the \defi{algebraic Brauer group} of \(X\), and the quotient \(\Br  X/\Br_1 X\) is called the \defi{transcendental Brauer group} of \(X\).  

			The transcendental Brauer group can also be studied using the Hochschild-Serre spectral sequence.  Indeed, the higher degree terms yield the following exact sequence (again using the assumption that \(k\) is a number field and so \(\HH^3(G_k, \kbar^{\times}) = 0\))~\cite{CTS-BrauerBook}*{5.24}
			\begin{equation}\label{eq:BrmodBr1}
			0 \to \frac{\Br  X}{\Br_1 X} \to \left(\Br  \Xbar\right)^{G_k} \to \HH^2(G_k, \Pic \Xbar).	
			\end{equation}
			Thus, if we have a good enough understanding of \(\Pic \Xbar\) and \(\Br  \Xbar\) as Galois modules, we can leverage~\eqref{eq:Br1} and~\eqref{eq:BrmodBr1} to compute the Brauer group modulo constants.  See~\cite{CTS-BrauerBook}*{Section 5.4} for more details on these approaches.

		\subsection{Properties of local evaluation maps}\label{sec:LocalEvaluation}
			If we have successfully computed the Brauer group (or the quotient \(\Br  X/\Br_0X\)), then the task that remains is describing the set \(X(\mathbb{A}_k)^{\alpha}\) for \(\alpha\) a set of generators of \(\Br  X/\Br_0 X\).

			Given an \(\alpha\in \Br  X\), there are some theoretic results that simplify calculation of the map \(\langle\alpha_v, -\rangle \colon X(k_v) \to \Q/\Z\).
			\begin{myprop}\label{prop:LocalEvaluation}
				Let \(X\) be a smooth projective variety over a number field \(k\), let \(v\) be a place of \(k\) and let \(\alpha\in \Br  X\).  
				\begin{enumerate}
					\item The map \(\langle\alpha_v, -\rangle \colon X(k_v) \to \Q/\Z\) is locally constant.~\cite{Poonen-Qpts}*{Prop. 8.2.9(a)}
					\item If \(\alpha_v\in \Br_0 X_v\), then the map \(\langle\alpha_v, -\rangle \colon X(k_v) \to \Q/\Z\) is constant.
					\item If \(v\) is a nonarchimedean place of good reduction (i.e., there exists a smooth proper model \(\mathscr{X} \to \OO_v\)) and \(\alpha_{k_v^{nr}} = 0\), then the map \(\langle\alpha_v, -\rangle \colon X(k_v) \to \Q/\Z\) is constant.~\cite{CTS-GoodReduction}*{Lemma 2.2(ii)}\label{part:LocallyConstant}
				\end{enumerate}
			\end{myprop}

			The vast majority of examples in the literature with \(X(\A_k)^{\{\alpha_i\}}=\emptyset\) actually have \(\langle(\alpha_i)_v, -\rangle \colon X(k_v) \to \Q/\Z\) constant for all \(i\) and for all places \(v\).  Based on this, one might be hopeful that constant evaluation everywhere is common, which would greatly simplify computation of the Brauer-Manin set.  However, these examples are misleading!  The examples in the literature were typically constructed to ensure this happens exactly so that the authors could prove that \(X(\A_k)^{\{\alpha_i\}}=\emptyset\).  In fact, experimentation shows that the expectation should be the opposite, namely:
			\begin{center}
				\textit{Whenever possible, there should be a place \(v\) where the image of \(\langle\alpha_v, -\rangle \colon X(k_v) \to \Q/\Z\) is large.}
			\end{center}

            Note that large image of \(\langle\alpha_v, -\rangle\) has implications for the Brauer-Manin set.
            \begin{mylem}
                Let \(X\) be a smooth projective variety over a number field \(k\) that has \(X(\A_k)\neq\emptyset\).  Let \(\alpha\in \Br  X\) be a nontrivial element.
                \begin{enumerate}
                    \item If there is a place \(v\) where the image of \(\langle\alpha_v, -\rangle \colon X(k_v) \to \Q/\Z\) has at least \(2\) elements, then \(X(\A_k)^{\alpha}\subsetneq X(\A_k)\).\label{part:nonconstant}
                    \item If there is a place \(v\) where the image of \(\langle\alpha_v, -\rangle \colon X(k_v) \to \Q/\Z\) is equal to \(\frac1{\textup{ord}(\alpha)}\Z/\Z\), then \(X(\A_k)^{\alpha}\neq\emptyset\).\label{part:surjective}
                \end{enumerate}
            \end{mylem}
			\begin{proof}
				Let \(v_0\) be a place satisfying the assumptions of part~\ref{part:nonconstant} or part~\ref{part:surjective} and let \((P_v)_{v\in \Omega}\in X(\A_k)\).  
                
                If \(v_0\) satisfies the assumptions of part~\ref{part:nonconstant}, then there exists a point \(Q_{v_0}\in X(k_{v_0})\) such that \(\langle\alpha_{v_0},P_{v_0}\rangle\neq \langle\alpha_{v_0},Q_{v_0}\rangle\).  Define the adelic point \((Q_v)_{v\in \Omega}\) to be equal to \((P_v)\) at all places \(v\neq v_0\) and to be equal to \(Q_{v_0}\) at \(v=v_0\).  Since \(\inv_v\) is an isomorphism, the difference
                \[
                  \sum_v \inv_v\langle\alpha,P_v\rangle - \sum_v\inv_v\langle\alpha,Q_v\rangle = \inv_v\left(\langle\alpha_{v_0},P_{v_0}\rangle- \langle\alpha_{v_0},Q_{v_0}\rangle\right) 
                \]
                must be nonzero, so at least one of \((P_v)\), \((Q_v)\) must not be orthogonal to \(\alpha\).  Hence, \(X(\A_k)^{\alpha}\subsetneq X(\A_k)\).

                Now assume that \(v_0\) satisfies the assumptions of part~\ref{part:surjective}.  Then, there exists a point \(Q_{v_0}\in X(k_{v_0})\) such that \(\inv_{v_0}\langle\alpha_{v_0},Q_{v_0}\rangle = \sum_{v\neq v_0} \inv_v\langle\alpha_v,P_v\rangle\).  Define the adelic point \((Q_v)_{v\in \Omega}\), as above, i.e., so that \(Q_v\) is equal to \((P_v)\) at all places \(v\neq v_0\) and is equal to \(Q_{v_0}\) at \(v=v_0\).  Then, by construction, \( (Q_v)\in X(\A_k)^{\alpha}\).
			\end{proof}
			The experimental observation articulated above is corroborated by several theoretic results, the earliest of which is due to Harari in 1994.
			\begin{myprop}[\cite{Harari-Duke}*{Prop. 6.1.1}]
				Let \(\pi\colon \mathcal{V} \to \PP^1_k\) be a projective morphism of smooth varieties, and let \(V\) denote the generic fiber of \(\pi\).  Assume that \(\mathcal{V}(\A_k)\neq\emptyset\) and that \(\pi\) has a geometric section.  
				
				If \(\Br  \mathcal{V}=\Br_0 \mathcal{V}\) and \(\Br  V\neq \Br_0V\), then there are infinitely many rational points \(\theta\in \PP^1(k)\) such that \(\mathcal{V}_{\theta}(\A_k)^{\Br} \subsetneq \mathcal{V}_{\theta}(\A_k)\).
			\end{myprop}
			The idea of the proof is to essentially show that for \(\alpha\in \Br  V - \Br  \calV\), there must be enough variation of the images of the local pairing with \(\alpha\) to obtain the result.

            Later work of Bright~\cite{Bright-BadReduction}*{Thm. 5.16} shows that if the residue field of \(k_v\) is sufficiently large, and the special fiber of a model \(\mathscr{X}\) over \(\OO_v\) of \(X/k_v\) has rich enough geometry, then there exists a Brauer class \(\alpha\in \Br X_v\) whose image is as large as possible.

            More recently, Pagano~\cite{Pagano}, following work of Bright and Newton~\cite{BN-OrderpElements}, constructed an example with surjective evaluation map at a place of good reduction, thereby showing that any generalization of Proposition~\ref{prop:LocalEvaluation}\eqref{part:LocallyConstant} must retain some assumption on \(\alpha\).

        \subsection{Extended exercise computing a Brauer-Manin obstruction}
        \begin{myexer}

		Let \(X\subset \PP^4\) be given by the vanishing of the following two quadrics
		\[
			st - x^2 + 5y^2, \quad (s+t)(s + 2t) - x^2 + 5z^2.
		\]
		This variety was first studied by Birch and Swinnerton-Dyer~\cite{BSD-dp4}.\label{BSD}
		\begin{enumerate}[itemsep=12pt]
			\item  Use the Hasse-Weil bounds to show that any smooth genus \(1\) curve over a finite field \(\F\) has an \(\F\)-point.  
            
            \item Show that an intersection of quadrics in \(\PP^3\) is a genus \(1\) curve, and prove that \(X\cap V(z)\) is smooth modulo \(p\) for all \(p\neq 2, 5\).  Use this to deduce that \(X(\Q_p)\neq \emptyset\) for all \(p\neq 2, 5\).
			\item Use the special case of Hensel's Lemma given in Exercise~\ref{exer:ConicLocal} to show that \(X(\Q_5)\neq\emptyset\) and that \(X(\Q_2)\neq\emptyset\).  Then prove that \(X(\A_{\Q})\neq \emptyset\).
			\item Using Exercise~\ref{exer:TrivialBrauerClass}, show that \((5, \frac{s}{t})\) and \((5, \frac{s+t}{s+2t})\) are trivial in \(\Br  \kk(X)\). 
			
			\item Using the previous part and Remark~\ref{rmk:ProductsOfQuaternions} show that, in \(\Br  \kk(X)\)
			\[
				\calA := \left(5, \frac{s+t}{s}\right) =  \left(5, \frac{s+2t}{s}\right) = 
				\left(5, \frac{s+t}{t}\right) =  \left(5, \frac{s+2t}{t}\right).
			\]
			Additionally show that for every point \(P \in X - V(s,t)\), there is an open set \(P\in U_P \subset X - V(s,t)\) such that at least one of \(\frac{s+t}{s}, \frac{s+t}{t}, \frac{s+2t}{s}, \frac{s+2t}{t}\) is regular and invertible on \(U\).  This proves that \(\alpha\in \Br U\).  Now use Theorem~\ref{thm:BrProps}\eqref{part:codim2} to prove that \(\alpha\in \Br X\).\label{Purity}
	
			\item Show that \(X(\A_{\Q})^{\calA} = \emptyset\) and hence \(X(\A_{\Q})^{\Br} = \emptyset\).  (\textit{Sketch}: First show that for all \(p\neq 5\) and \(P_p\in X(\Q_p)\), at least one of \(\frac{s+t}{s}, \frac{s+t}{t}, \frac{s+2t}{s}, \frac{s+2t}{t}\) is a \(p\)-adic unit at \(P_p\).  Then, noting that \(\Q_p(\sqrt{5})/\Q_p\) is unramified for \(p\neq 5\), use Exercise~\ref{exer:TrivialBrauerClass} to deduce that \(\calA(P_p)=0\in \Br  \Q_p\). Lastly, show that \(\calA(P_5)\neq 0\in \Br  \Q_5\) for all \(P_5\in X(\Q_5)\).)
		\end{enumerate}
    \end{myexer}
	\section{Capturing the Brauer-Manin obstruction to the existence of points}\label{sec:Capturing}

		We saw in Section~\ref{sec:LocalEvaluation} that ``unless forced otherwise'' we should typically expect \(X(\A_k)^{\alpha} \subsetneq X(\A_k)\) and in fact, that Brauer elements should not only cause obstructions, but cause \emph{independent} obstructions as much as they are able.  This shows that we cannot hope to meaningfully decrease the work involved to compute \(X(\A_k)^{\Br}\), that to compute the Brauer-Manin set, we really need to determine all the local pairings for a complete set of generators of \(\Br  X/\Br_0 X\).

		However, if we are merely interested in the existence of rational points (rather than finer information about them), we can make do with answering the possibly weaker decision problem: determining whether \(X(\A_k)^{\Br}\neq\emptyset\).

		Why may this problem be easier?  One can show that a consequence of the topological properties of the Brauer-Manin pairing (Proposition~\ref{prop:LocalEvaluation}) and the compactness of \(X(\A_k)\)\footnote{Compactness of \(X(\A_k)\) follows from the compactness of \(X(k_v)\) for local fields~\cite{Poonen-Qpts}*{Prop. 2.6.1}, the definition of the topology on \(X(\A_k)\)~\cite{Poonen-Qpts}*{Section 2.6.3}, and Tychonoff's theorem.} is that if \(X(\A_k)^{\Br} = \emptyset\) then this is witnessed by a finite collection of Brauer classes.

		\begin{myexer}\hfill

		\begin{enumerate}
			\item Using Proposition~\ref{prop:LocalEvaluation}\eqref{part:LocallyConstant}, show that \(\textup{ev}_{\alpha}\colon X(\A_k)\to \Q/\Z\) is locally constant.
			\item Let \(\alpha\in \Br  X\).  Show that \(X(\A_k)^{\alpha}\) is open and closed in the adelic topology.  Conclude that \(X(\A_k)^{\Br}\) is closed.
			\item Assume that \(X(\A_k)^{\Br} = \emptyset\).  Show that there exists a finite set \(B\subset \Br  X\) such that \(X(\A_k)^B = \emptyset.\)
		\end{enumerate}
	\end{myexer}

	In this section we consider the question of whether any properties of this finite subset of Brauer classes can be determined \textit{a priori}, i.e., without computing \(X(\A_k)^{\Br}\).  With this goal in mind, we make the following definitions.
	\begin{mydef} Let \(X\) be a smooth projective variety over a number field \(k\) and let \(\mathcal{B}\subset \Br  X\).
		\begin{enumerate}
			\item The set \(\mathcal{B}\) \defi{captures the Brauer-Manin obstruction} (to the existence of rational points) if \(X(\A_k)^{\Br} = \emptyset\Rightarrow X(\A_k)^{\mathcal{B}} = \emptyset.\)
			\item The set \(\mathcal{B}\) \defi{completely captures the Brauer-Manin obstruction} (to the existence of rational points) if, for all \(\mathcal{B}'\subset \Br  X\), we have 
            \[
                X(\A_k)^{\mathcal{B}'} = \emptyset\Rightarrow X(\A_k)^{\mathcal{B}\cap\mathcal{B}'} = \emptyset.\footnote{This terminology came out of discussions with Brendan Creutz.  We use `capture' in the sense of ``record or express accurately''}
            \]
		\end{enumerate}
	\end{mydef}
	Although these definitions are relatively recent, prior results can be phrased in this language.
	\begin{mythm}\label{thm:ClassicalCapturing}\hfill
		\begin{enumerate}
			\item If \(C\) is a smooth projective degree d genus \(1\) curve, then \(\Br  C[d^{\infty}]\) completely captures the Brauer-Manin obstruction.  (Follows from~\cite{Manin-ICM})
			\item If \(X\) is a smooth projective cubic surface (equivalently, a degree \(3\) del Pezzo) then \(\Br  X[3]\) completely captures the Brauer-Manin obstruction.~\cite{SD-Cubic}*{Corollary 1}
			\item If \(X\) is a smooth projective del Pezzo surface of degree \(3\) or \(4\), then there exists an \(\alpha\in \Br  X\) such that \(\{\alpha\}\) captures the Brauer-Manin obstruction.~\cite{CTP}*{Lemma 3.4 and the following Remark 2}
		\end{enumerate}
	\end{mythm}

    Given a class of varieties, one could hope to find a \(\mathcal{B}\subset \Br  X\) that captures or completely captures the Brauer-Manin obstruction where \(X(\A_k)^{\mathcal{B}}\) is (more) easily computable.  Recent work of Skorobogatov--Zarhin, Creutz and myself, and Nakahara indicates this is possible for varieties closely related to abelian varieties.

    \begin{mythm}\label{thm:RecentCapturing}\hfill
        \begin{enumerate}
            \item {\cite{CV-Capturing}*{Thm. 1.2}} Let \(V/k\) be a torsor under an abelian variety.  Then \(\Br  V[\textup{per}(V)^{\infty}]\) completely captures the Brauer-Manin obstruction.\label{part:CV}
            \item \cite{CV-Capturing}*{Thm. A.1} (building on~\cites{SZ-Capturing, CV-Capturing}) Let \(Y/k\) be a Kummer variety.  Then \(\Br  V[2^{\infty}]\) completely captures the Brauer-Manin obstruction.\label{part:Kummer}
            \item \cite{Nakahara-Capturing}*{Thm. 1.5} Let \(\pi\colon X \to \PP^n\) be a fibration of torsors under abelian varieties and let \(P\) denote the period of the generic fiber \(X_{\eta}\) (considered as a torsor under an abelian variety over \(\kk(\eta)\)).  Then \(\Br  X[P^{\infty}]\) completely captures the Brauer-Manin obstruction.\label{part:Nakahara}
            \item \cite{Creutz-Capturing}*{Thm. 1} Let \(V/k\) be a torsor under an abelian variety.  Then the subgroup of locally constant Brauer classes completely captures the Brauer-Manin obstruction.\label{part:Creutz}
        \end{enumerate}
    \end{mythm}

    However, it seems unlikely that there exists an easily describable proper subgroup \(\mathcal{B}\) that captures or completely captures the Brauer-Manin obstruction for arbritrary varieties.  In joint work with Creutz and Voloch, we pursued this line of inquiry for curves~\cite{CVV-CapturingCurves}.  While our results do not rule out the possibility of a capturing or completely capturing proper subgroup, they do suggest that any such subgroup likely does not depend only on the genus and the degree.  

    \subsection{Open questions}

    \subsubsection{del Pezzo surfaces}  The older results summarized in Theorem~\ref{thm:ClassicalCapturing} show that for cubic and quartic del Pezzo surfaces, there is always a single Brauer class that captures the Brauer-Manin obstruction.  For degree \(2\) del Pezzo surfaces (the only other case that can have empty Brauer set and nonempty adelic points), Nakahara's result from Theorem~\ref{thm:RecentCapturing}\eqref{part:Nakahara} implies that the \(2\)-primary subgroup of the Brauer group completely captures the Brauer-Manin obstruction.  However, it is not known how many elements are needed to detect a Brauer-Manin obstruction.  Corn constructed an example of a degree \(2\) del Pezzo surfaces where two Brauer classes are required to detect the obstruction~\cite{Corn-dp2}*{Example 9.4}.  However, the \(\F_2\)-dimension of \(\Br  X/(\Br_0 X + 2\Br  X)\) can be as large as \(6\)~\cite{Corn-dp2}*{Thm. 4.1}.  Are there degree \(2\) del Pezzo surfaces that require \(3\) Brauer classes to detect the obstruction? \(4\) Brauer classes? \(6\) Brauer classes?

    \subsubsection{K3 surfaces} As mentioned in Theorem~\ref{thm:RecentCapturing}\eqref{part:Kummer}, Skorobogatov proved that the \(2\)-primary Brauer-Manin obstruction completely captures the Brauer-Manin obstruction on Kummer surfaces.  Since Kummer surfaces are a particular class of K3 surfaces, it is natural to ask whether the \(2\)-primary subgroup captures the Brauer-Manin obstruction for all K3 surfaces.  Unfortunately, this is false, as work of Gvirtz, Loughran, and Nakahara shows~\cite{GLN}*{Thm. 1.5}.  In addition, Berg and V\'arilly-Alvarado construct a K3 surface of Picard rank \(1\) where the \(2\)-primary subgroup fails to completely capture the Brauer-Manin obstruction~\cite{BVA}*{Thm. 1.1}.  In both of these papers, the authors show that the obstruction can be detected by \(3\)-torsion.  Thus, one can ask whether the \(6\)-primary subgroup captures (or completely captures) the Brauer-Manin obstruction on K3 surfaces or whether classes of order prime to \(6\) can also obstruct the existence of rational points.

	\section{Behavior of the Brauer-Manin obstruction over extensions}\label{sec:Persistence}

		Thus far, we have considered the question of determining whether \(X(k)\neq\emptyset\) or whether \(X(\A_k)^{\Br}\neq \emptyset\) over a fixed ground field \(k\).  However, the arithmetic content of \(X\) is much richer than that.  In some sense, the arithmetic of \(X\) is the set of all algebraic points \(X(\kbar)\) together with its action of the absolute Galois group \(G_k := \Gal(\kbar/k)\).

		Thus, we can ask: what can we say about the sets \(X(\A_L)^{\Br}\) as \(L\) ranges over extensions of \(k\)?  Is knowledge of \(X(\A_k)^{\Br}\) enough to give us some information about what happens when we range over algebraic extensions?

		As we saw in Section~\ref{sec:LocalEvaluation}, asking to completely describe the Brauer sets \(X(\A_L)^{\Br}\) is likely too fine of a question since ``unless there is a reason otherwise'' we expect Brauer classes to carve out a proper subset of adelic points.  This is true even when \(X\) stays fixed and only the field varies, as the following results of Liang and Nakahara--Roven demonstrate.

		\begin{mythm}[\cite{Liang-WA}*{Thm. 3.1}]
			For all number fields \(k\), there exists a Ch\^atelet surface \(V/k\) and a finite extension \(L/k\) such that 
			\[
				V(\A_k)^{\Br} = V(\A_k)\neq\emptyset, \quad\textup{and}\quad
				V(\A_L)^{\Br}\subsetneq V(\A_L).
			\]
		\end{mythm}

		\begin{mythm}[\cite{NR-WA}*{Thm. 1.1}]
			Let \(k\) be a number field, let \(a\in k^{\times}\) and let \(p(x)\) be a degree \(4\) separable polynomial over \(k\).  Let \(X\) be the Ch\^atelet surface given by \(y^2 - az^2 = p(x)\) and let\(L\) be an extension where \(p(x)\) splits completely and \(X(\A_L)\neq\emptyset\).  Then \(X(\A_L)^{\Br}\subsetneq X(\A_L)\) if \(a\) is not totally positive or if there exists a finite place \(v\) such that \(p(x)\bmod v\) is not separable and \(a\notin k_v^{\times2}\).
		\end{mythm}

		In light of these results, we must modify our question of interest.  Instead of asking to describe the Brauer-Manin sets over all extensions, we can instead ask whether we can say anything about the set of field extensions
		\[
		\{L/k \textup{ finite}:X(\A_L)^{\Br} = \emptyset\}	
		\]
		or its complement 
		\[
		\{L/k \textup{ finite}:X(\A_L)^{\Br} \neq \emptyset\}	
		\]
		based on information about \(X\) over \(k\).

        This is a relatively new research direction, and, as we will see, the results thus far are mainly limited to geometrically rational surfaces.  More exploration is needed to determine whether the phenomena we will see manifests in greater generality.
		\subsection{Compatibility of the Brauer-Manin pairing with corestriction}

			The functoriality of the Brauer group and compatibility of the Brauer-Manin pairing with corestriction leads to the following containments of Brauer sets under extensions.
			\begin{mylem}[\cite{CV-QuadraticPoints}*{Lemma 3.1}]\label{lem:FunctorialityCor}
				Let \(K/k\) be an extension of number fields, let \(Y/k\) be a smooth projective geometryically integral variety over \(k\), and let \(B\subset \Br  Y_K\).  Then \(Y(\A_k)^{\textup{Cor}_{Y_K/Y_k}(B)}\subset Y(\A_K)^{B}\).  In particular,
				\begin{enumerate}
					\item if \(Y(\A_k)^{\Br}\neq\emptyset\) then \(Y(\A_K)^{\Br}\neq\emptyset\), and\label{part:BrSetContainment}
					\item for any \(d|[K:k]\), we have \(Y(\A_k)\subset Y(\A_K)^{(\Br  Y)[d]}\).\label{part:AdelicPointsContainment}
				\end{enumerate}
			\end{mylem}
			\begin{myrmk}
				In an early draft of~\cite{CV-QuadraticPoints}, Creutz and I proved parts~\eqref{part:BrSetContainment} and~\eqref{part:AdelicPointsContainment} in our case of interest, quartic del Pezzo surfaces.  Upon review of that draft, Wittenberg observed that the lemma held more generally and outlined a proof of the version given here.
			\end{myrmk}
			From Lemma~\ref{lem:FunctorialityCor}, we can immediately deduce that if \(X(\A_k)^{\Br}\neq\emptyset\), then 
			\[
				\{L/k \textup{ finite}:X(\A_L)^{\Br} = \emptyset\} = \emptyset \quad\textup{and}\quad
				\{L/k \textup{ finite}:X(\A_L)^{\Br} \neq \emptyset\} = \{L/k \textup{ finite}\}.
			\]
			Thus, the main case of interest for this question is when \(X(\A_k)^{\Br} = \emptyset\).

			\subsection{Vanishing of the Brauer-Manin obstruction}
				In this section we retain the assumption that \(X(\A_k)^{\Br} = \emptyset\) and consider sufficient conditions for the Brauer-Manin obstruction to \emph{vanish} over the extension \(L/k\), i.e., such that \(X(\A_L)^{\Br}\neq\emptyset\).

				From Lemma~\ref{lem:FunctorialityCor}, we can deduce that if \(X(\A_k)\neq\emptyset\) \emph{and} the \([L:k]\)-torsion subgroup of \(\Br  X_L\) captures the Brauer-Manin obstruction, then the Brauer-Manin obstruction vanishes over \(L\).  Here are some example cases where these properties are known.

				\begin{myprop}[Corollary of Theorem~\ref{thm:ClassicalCapturing} and Lemma~\ref{lem:FunctorialityCor}]\label{prop:VanishingBM}\hfill
					\begin{enumerate}
						\item Let \(\pi\colon X \to \PP^1_k\) be an everywhere locally soluble conic bundle.  Then there exists a finite extension \(K/k\) such that for all even degree extensions \(L/k\) with \(L\) linearly disjoint from \(K\), we have \(X(\A_L)^{\Br}\neq\emptyset\).
						\item Let \(X \subset \PP^3_k\) be an everywhere locally soluble cubic surface.  Then there exists a finite extension \(K/k\) such that for all extensions \(L/k\) with \(L\) linearly disjoint from \(K\) and \([L:k]\equiv 0 \bmod 3\), we have \(X(\A_L)^{\Br}\neq\emptyset\).
						\item  Let \(X \subset \PP^4_k\) be an everywhere locally soluble quartic del Pezzo surface.  Then there exists a finite extension \(K/k\) such that for all even degree extensions \(L/k\) with \(L\) linearly disjoint from \(K\), we have \(X(\A_L)^{\Br}\neq\emptyset\).
					\end{enumerate}
				\end{myprop}
				More generally, we can obtain a result of this flavor whenever \(X\) is a variety whose geometric Picard group is torsion free and whose geometric Brauer group is trivial (see also~\cite{Kanevsky}).

				For the varieties in Proposition~\ref{prop:VanishingBM}, the Brauer-Manin obstruction is conjecturally the only obstruction to the existence of rational points.  Thus, Proposition~\ref{prop:VanishingBM} suggests the existence of points over numerous field extensions.  What can we prove unconditionally?

				\begin{myprop}
					Let \(\pi\colon X\to \PP^1\) be an everywhere locally soluble conic bundle.  Then, for any finite set of places \(S\subset \Omega_k\) and any collection of quadratic etale algebras \((F_v)_{v\in S}\), there exists a quadratic extension \(L/k\) such that \(L\otimes_{k}k_v \simeq F_v\) and \(X(L)\neq\emptyset\).
				\end{myprop}
				\begin{proof}
					By the implicit function theorem, for any \(x_v\in X(k_v)\) there exists an open set \(U_v\subset \PP^1(k_v)\) such that \(\pi(x_v)\subset U_v \subset \pi(X(k_v))\).  By weak approximation on \(\PP^1\), there exists a \(t\in \PP^1(k)\) such that \(t\in U_v\) for all \(v\in S\).  In particular, the fiber \(X_t\) is a conic that has \(k_v\)-points for all \(v\in k\).  Let \(T:=\{v\in \Omega_k :  X_t(k_v) = \emptyset\}\); note that \(T\) is finite (see Proposition~\ref{prop:LocalSolubilityFiniteComputation}) and disjoint from \(S\).  By~\cite{Neukirch-BonnLectures}*{Part II, Def. 1.3 and Thm. 5.6}, \(\inv_F\circ \textup{Res}_{F/k_v} = [F:k_v] \circ \inv_{k_v}\) for any extension of local fields \(F/k_v\).  Therefore, for any quadratic extenion \(K/k\), \(X_t(\A_K) \neq \emptyset\) if and only if \(v\) does not split in \(K\) for all \(v\in T\).  Since \(S\) and \(T\) are disjoint and finite, there exists a quadratic extension \(L/k\) that has the desired splitting behavior at all \(v\in S\cup T\).   Since conics satisfy the local-to-global principle and \(X_t(\A_L)\neq\emptyset\), this proves that \(X_t(L)\neq\emptyset\), as desired. 
				\end{proof}

				Analogous arguments can prove an analogous proposition for cubic surfaces (where quadratic extensions are replaced by cubic extensions).  For locally soluble quartic del Pezzo surfaces, we can unconditionally prove the existence of a quadratic point~\cite{CV-QuadraticPoints}*{Proposition 4.7 and Remark 4.8}; however, it is a nonexplicit construction and doesn't easily give control of local behavior.  Thus, our unconditional results over number fields fall far short of what is predicted by Proposition~\ref{prop:VanishingBM}.

				If \(X(\A_k) = \emptyset\), then Lemma~\ref{lem:FunctorialityCor} no longer can be leveraged to show that the Brauer-Manin obstruction vanishes over an extension.  Indeed, the lack of adelic points over \(k\) limits the construction of adelic points over the extension, and so it is intrinsically more difficult to prove that the obstruction vanishes.  Despite this added difficulty, there are some positive results in this direction.

			\begin{mythm}[\cite{Roven}]
				Let \(\pi\colon X \to \PP^1_k\) be a conic bundle.  Then there exists a finite extension \(K/k\) such that for all quadratic \(L/k\) linearly disjoint from \(K\), we have
				\[
					X(\A_L) \neq\emptyset\; \Leftrightarrow\; X(\A_L)^{\Br}\neq\emptyset.	
				\]
			\end{mythm}

			\begin{mythm}[\cite{Roven}]\label{thm:RovenChatelet}
				Let \(\pi\colon X \to \PP^1_k\) be a Ch\^atelet surface.  Then there exists a finite extension \(K/k\) such that for all even degree \(L/k\) linearly disjoint from \(K\), we have
				\[
					X(\A_L) \neq\emptyset\; \Leftrightarrow\; X(\A_L)^{\Br}\neq\emptyset.	
				\]
			\end{mythm}	
            \begin{myrmk}
                Colliot-Th\'el\`ene, Sansuc, and Swinnerton-Dyer~\cites{CTSSDI, CTSSDII} have proved that for Ch\^atelet surfaces over number fields, the set of \(k\)-rational points are dense in the Brauer-Manin set.  Using this one can deduce from Theorem~\ref{thm:RovenChatelet} that for all even degree \(L/k\) linearly disjoint from \(K\), \(X\) satisfies the local-to-global principle over \(L\).
            \end{myrmk}
			\begin{mythm}[Special case of~\cite{CV-QuadraticPoints}*{Thm. 1.2}]
				Let \(X\subset \PP^4\) be a quartic del Pezzo surface over a number field \(k\).  Then for all places \(v\), there exists a quadratic extension \(L_w/k_v\) such that \(X(L_w) \neq \emptyset\).  Furthermore, if all rank \(4\) quadrics containing \(X\) are defined over \(k\), then there exists a quadratic \(L/k\) such that \(X(\A_L)^{\Br} \neq \emptyset\).
			\end{mythm}

		\subsection{Persistence of a Brauer-Manin obstruction}
			In this section we study the complementary question of conditions on \(L/k\) that guarantee that a Brauer-Manin obstruction persists, i.e., that \(X(\A_k)^{\Br} = \emptyset\) implies that \(X(\A_L)^{\Br} = \emptyset\).  We should think of this as a harder condition to achieve, because it typically becomes easier to obtain points over extensions.  

			However, there are some simple varieties where we can prove persistence.
			\begin{mythm}[Corollary of Springer's theorem~\cite{Springer}] Let \(X\) be a smooth quadric over a number field \(k\).  Then a Brauer-Manin obstruction persists over any odd degree extension, i.e., \(X(\A_k)^{\Br} = \emptyset\) implies that \(X(\A_L)^{\Br}=\emptyset\) for all odd degree extensions \(L/k\).
			\end{mythm}

			\begin{mythm}[Corollary of~\cite{CTC}*{Thm. C} and~\cite{CTP}*{Remark 3}] Let \(X\) be a quartic del Pezzo surface over a number field \(k\).  Then a Brauer-Manin obstruction persists over any odd degree extension, i.e., \(X(\A_k)^{\Br} = \emptyset\) implies that \(X(\A_L)^{\Br}=\emptyset\) for all odd degree extensions \(L/k\).
			\end{mythm}

			\begin{mythm}[\cite{RV-cubics}*{Thm. 1.1}] Let \(X\) be a smooth cubic surface over a number field \(k\).  Then a Brauer-Manin obstruction persists over any extension with degree coprime to three, i.e., \(X(\A_k)^{\Br} = \emptyset\) implies that \(X(\A_L)^{\Br}=\emptyset\) for all \(L/k\) with \(3\nmid [L:k]\).
			\end{mythm}


%
%

\end{document}